\documentclass[reqno, 12pt]{amsart}
\title[Optimal H\"older regularity of SLE]{Optimal H\"older exponent
  for the SLE path}
\author[Johansson]{Fredrik Johansson}
\address{Fredrik Johansson\\
Department of Mathematics\\
The Royal Institute of Technology\\
S -- 100 44 Stockholm\\
SWEDEN}
\thanks{Johansson is supported by grant KAW 2005.0098 from the Knut and Alice Wallenberg
Foundation.}
\email{frejo@math.kth.se}

\author[Lawler]{Gregory F. Lawler}
\address{Gregory F. Lawler\\ 
Department of Mathematics and Department of Statistics\\
University of Chicago\\
5734 S. University Avenue\\
Chicago, IL 60637, USA}
\thanks{Lawler is supported by National Science
Foundation  grant DMS-0734151.}
\email{lawler@math.uchicago.edu}
\usepackage{amscd}
\usepackage{amsthm,amsmath,amsfonts,amssymb}
\usepackage{verbatim}
\usepackage{lmodern}
\usepackage{graphicx}

\usepackage[T1]{fontenc}              
\newcommand{\imag}{\operatorname{Im} \,}
\newcommand{\real}{\operatorname{Re} \,}

\newcommand{\iy}{\infty}
\newcommand{\fr}{\frac}

\newcommand{\HH}{\mathbb{H}}
\newcommand{\DD}{\mathbb{D}}
\newcommand{\CC}{\mathbb{C}}

\newcommand{\NN}{\mathbb{N}}
\newcommand{\EE}{\mathbb{E}}
\newcommand{\E}{\mathbb{E}}
\newcommand{\PP}{\mathbb{P}}
\newcommand{\Prob}{\mathbb{P}}

\newcommand{\vp}{\varphi}
\newcommand{\hcap}{\operatorname{hcap}}
\newcommand{\diam}{\operatorname{diam}}
\newcommand{\height}{\operatorname{height}}
\newcommand{\dist}{\operatorname{dist}}
\newcommand{\SLE}{\operatorname{SLE}}
\def \p {\partial}
\def \phi {{\varphi}}

\newtheorem{Theorem}{Theorem}[section]

\newtheorem{Proposition}[Theorem]{Proposition}
\newtheorem{Corollary}[Theorem]{Corollary}
\newtheorem{Lemma}[Theorem]{Lemma}

\theoremstyle{definition}
\newtheorem{Definition}[Theorem]{Definition}

\theoremstyle{remark}
\newtheorem*{Remark}{Remark}

\begin{document}
\bibliographystyle{alpha}

\begin{abstract}
We prove an upper bound on the optimal H\"older exponent for
the chordal $\SLE$ path parameterized by capacity and thereby establish
the optimal exponent as conjectured by J. Lind. We also give
a new proof of the lower bound. Our proofs are based on the
sharp estimates of 
moments of the derivative of the inverse map.  In particular,
we improve an estimate of the second author.
\end{abstract}
\maketitle
\section{Introduction}
The Schramm-Loewner evolution, or $\SLE(\kappa)$, is a one-parameter
family of random fractal curves that was introduced by O. Schramm in 
\cite{Schramm_LERW} as a candidate for the scaling limit of the
loop-erased random walk. Since then, $\SLE$ has been shown to
describe the scaling limits of a number of discrete models from
statistical physics and to provide tools for their
rigorous understanding. The properties of $\SLE$ curves have been
studied by a number of authors. For instance, S. Rohde
and Schramm proved in \cite{Rohde_Schramm} the existence and
continuity of the path and gave an upper bound on the Hausdorff
dimension. V. Beffara \cite{Beffara_dimension} proved a lower bound on
Hausdorff dimension and thus showed that the dimension is almost
surely the minimum of $1+\kappa/8$ and $2$. J. Lind \cite{Lind}
improved the estimates by Rohde and Schramm and proved that the $\SLE(\kappa)$ path is almost surely
H\"older continuous. She conjectured that the H\"older exponent she obtained
is the optimal one, that is, that the exponent is the largest possible. In this paper we prove this conjecture. More precisely, we
prove the following theorem. Let $\kappa \ge 0$
and set
\[ 
\alpha_*=\alpha_*(\kappa) =    1 - \frac{\kappa} 
{24 + 2\kappa - 8 \sqrt{8 + \kappa}}, \quad  \alpha _0 =
\min\left\{\frac{1}{2},\, \alpha_* \right\}.
\] 
\begin{Theorem}  \label{main.thm} 
Let $\gamma(t)$ be the chordal $\SLE (\kappa)$ path parameterized by
half-plane capacity. With probability one the following
holds.
\begin{itemize}
\item
 $\gamma(t), \, t \in [0,1],$ 
is H\"older continuous of order $\alpha$ for 
$\alpha < \alpha_0$ and is not H\"older continuous 
of order $\alpha$ for $\alpha > \alpha_0.$ 
\item For every $0 < \epsilon < 1$,  $\gamma(t), \, t \in [\epsilon,1],$ 
is H\"older continuous of order $\alpha$ for 
$\alpha < \alpha_*$ and is not H\"older continuous 
of order $\alpha$ for $\alpha > \alpha_*.$ 
\end{itemize}
\end{Theorem} 
\begin{figure}[t]
\centering
\includegraphics[width=85 mm]{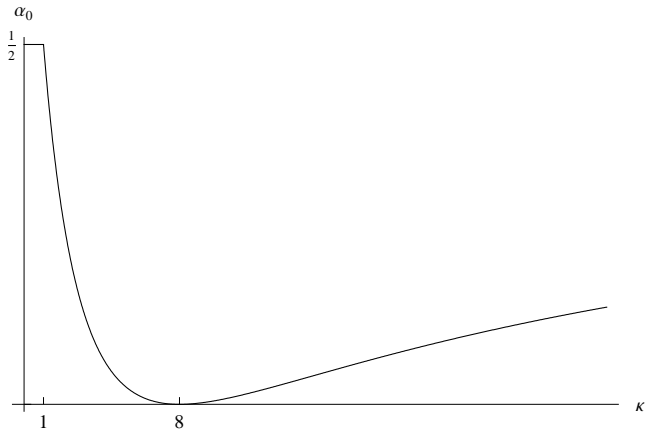}
\caption{Optimal H\"older exponent as a function of $\kappa$.}
\end{figure}
As mentioned, the lower bound on the optimal exponent was proved in \cite{Lind}. The upper
bound is new and we also give a new proof of the lower bound. The
phase transition of $\alpha_0(\kappa)$ at $\kappa=1$ is due to the
geometry of the path at the base; roughly, in the capacity
 parameterization the path is
H\"older-$1/2$ at $t=0$ for all $\kappa \ge 0$. 

The main tool needed is sharp estimates for the moments
of the derivative close to the preimage of the tip of the growing curve.
We will use the work in \cite{Lawler_multifractal} and also \cite{Johansson_Lawler_tip},
but we will have to extend one of these results in this paper.
  We
will use these results as a ``black box'', more or less, in the
main part of the paper and then give a derivation in the final section,
building on the argument in \cite{Lawler_multifractal}. To be
more specific, 
  we need to control the growth of the derivative as the
preimage of the tip is approached radially. To get the
lower bound on the exponent estimates on the derivative from above are
needed, and to get the
upper bound on the optimal exponent, one has to control second moments and time
correlations.

Our paper is organized as follows. In Section \ref{Preliminaries} we
give definitions and discuss some well-known results. In
Section \ref{Deterministic} we prove a number of results about
Loewner chains that are not particular to $\SLE$, but rather to
Loewner chains driven by functions which are weakly H\"older-1/2,
see \eqref{mod_of_cont} for a definition. Notably, we give
estimates on the modulus of continuity of the curve in terms of the
radial growth of the derivative. In the Section \ref{upsec}
we state the basic moment estimates and use these
  together with the results from Section
\ref{Deterministic} to prove Theorem \ref{main.thm}.  Finally,
Section \ref{momsec} proves the estimates on the derivative.
Here we review without proof the main results from  \cite{Lawler_multifractal} 
that we need, and then establish the new estimates.

\section{Preliminaries}\label{Preliminaries}
Let $U_t, \, t \ge 0$ be a continuous real-valued function. We
shall consider chordal Loewner chains $(g_t)$, that is, solutions to the chordal Loewner equation
\begin{equation}
\label{LODE}
\partial_t g_t=\fr{a}{g_t-U_t},\quad t>0, \quad g_0(z)=z,
\end{equation}
for $a >0$ fixed. We define the associated continuously growing hull
\[
K_t=\{z \in \HH: \tau(z) \le t\},
\]
where $\tau(z)$ is the blow-up time of \eqref{LODE}. For each $t >0$
the function $z \mapsto g_t(z)$ maps $H_t:=\HH \setminus K_t$
conformally onto $\HH$ and the inverse mapping $f_t:=g_t^{-1}$
satisfies the partial differential equation
\begin{equation}
\label{LPDE}
\partial_t f_t=-f'_t\fr{a}{z-U_t},\quad f_0(z)=z.
\end{equation}
Throughout the paper we will use the notation 
\begin{equation*}
\hat{f}_t(z):=f_t(U_t+z),
\end{equation*}
for $z \in \HH$. 

The time-reversed Loewner equation
\begin{equation}
\label{BLODE}
\partial_t F_t=-\fr{a}{F_t-U_{T-t}},\quad t \in (0,T],\quad F_0(z)=z
\end{equation}
is often useful to avoid dealing directly with (\ref{LPDE}): it is
easy to see that if $F$ is a solution \eqref{BLODE} and $f$ a solution to \eqref{LPDE} then
\[
F_T(z)=f_T(z).
\]
If there is a curve $\gamma(t)$ such that $H_t$ is the unbounded
connected component of $\HH \setminus \gamma[0,t]$ we say that the Loewner
chain $(g_t)$ is generated by a curve. This property is known \cite[Theorem
4.1]{Rohde_Schramm} to be equivalent to the existence of the radial limit
\begin{equation}\label{feb10}
\lim_{y \to 0}\hat f_t(iy)=:\gamma(t)
\end{equation}
for each $t >0$ together with continuity of $t \mapsto
\gamma(t)$. Loewner chains corresponding to driving functions with
strictly higher regularity than H\"older-1/2 are always generated by
a simple
curve; in fact H\"older-1/2 continuity with a sufficiently
small norm is sufficient to guarantee a simple
curve. Conversely, there are examples of Loewner chains corresponding to
H\"older-1/2 functions (with large norm) that are not generated by
a curve, see \cite{Marshall_Rohde} for these results.

In particular we will be interested in Schramm-Loewner evolution, 
$\SLE (\kappa)$, defined as the Loewner chain
corresponding to $a=2$ and $U_t=\sqrt{\kappa}B_t$, where $B$ is standard
Brownian motion and $\kappa \ge 0$. $\SLE$ is known to be generated by
a curve. We note that the $\SLE
(\kappa)$ path is simple if and only if $0 \le \kappa \le
4$ and space filling for $\kappa \ge 8$, see \cite{Rohde_Schramm} for
proofs of these facts. It may be noted that there is presently no known direct
proof of that $\SLE (8)$ is generated by a curve, see \cite{LSW1} for
an indrect proof.  

\begin{Definition}
A {\em (positive)
subpower function} is a continuous,
 non-decreasing function $\varphi:[0, \infty)
\to (0, \infty)$ 
that satisfies
\[
\lim_{x \to \infty} x^{-\nu} \varphi(x) = 0
\]  
for all $\nu >0$. 
\end{Definition}
\section{Deterministic results}\label{Deterministic}
In this section we prove a number of results about Loewner chains and
associated curves that are not special to $\SLE$. In particular we
consider Loewner chains corresponding to driving functions which are
weakly H\"older-1/2, see \eqref{mod_of_cont}.
 
For convenience we state the following well-known
result, see \cite{Pommerenke_bb} for a proof.
\begin{Lemma}[The Koebe distortion and one-quarter theorems]
Suppose $f: D \to \CC$ is a conformal map and set
$d=\dist(z, \partial D)$ for $z \in D$. Then
\begin{equation}
\label{distortion}
\fr{1-r}{(1+r)^3}|f'(z)| \le |f'(w)| \le \fr{1+r}{(1-r)^3}|f'(z)|, \quad |z-w| \le rd,
\end{equation}
and 
\begin{equation}
\label{quarter}
B(f(z), d|f'(z)|/4) \subset f(D),
\end{equation}
where $B(w,\rho)$ denotes the open disk of radius $\rho$ around $w$.
\end{Lemma}
\begin{Lemma}
\label{whitney}
Let $S$ be the rectangle $S = \{x+iy: -1 < x < 1, 0 < y < 1\}$.
There exist $c,\alpha < \infty$ such that if 
$f$ is a conformal map defined on $2S$, 
$z,w \in S$ and $\imag z, \, \imag w
\ge 1/r$, then 
\begin{equation}   
|f'(z)| \le c r^{\alpha} |f'(w)|.
\end{equation}
\end{Lemma}
\begin{proof}  One way to prove this is to
take a conformal transformation of $2S$ onto the unit disk with
$f(z) = 0$ and using the distortion theorem.
A more direct approach is as follows.
We may assume that $\imag z = 1/r$. Take a Whitney decomposition of
$S$, that is, a partition of $S$ into dyadic rectangles $\{S_{j,k}\}$
where
\[S_{j,k}=\{x+iy:j2^{-k} \le x \le
(j+1)2^{-k}, \, 2^{-(k+1)} \le y \le 2^{-k}\},\]
for $j=0, \ldots, 2^k-1,$ and $k \in \NN$. Let $u,v$ be in the same
rectangle. Then by iterating
(\ref{distortion}) it follows that there exists
a constant $c_1$ (uniform for all rectangles; $12^5$ would work for instance) such that
\begin{equation}\label{comparable}
c_1^{-1} \le |f'(u)|/|f'(v)| \le c_1.
\end{equation}
Suppose that $z \in S_{j,k}$. Then $k \le \log r \le k+1$. 
It follows that there exists a path in $S$ (the hyperbolic geodesic,
for instance)
that connects $z$ to $w$ and
intersects at most $2 \lceil \log r \rceil+2$ rectangles. Hence, there is a constant $c_2$ such that by
iterating (\ref{comparable}) along the path at most $c_2(\log r+1)$ times
\begin{align*}
|f'(z)| & \le c_1^{c_2 (\log r+1)}|f'(w)| \\
&=c r^{\alpha}|f'(w)|,
\end{align*}
for $\alpha = c_2 \log c_1$ and $c=c_1^{c_2}$.
\end{proof}

A proof of the next lemma may be found in \cite{Lawler_cip}. For a set
$K$, we let
$\diam(K)$ denotes the diameter of $K$ and $\height(K):=\sup\{\imag z:
z \in K\}$. 
\begin{Lemma}
\label{2.1}
Suppose $K_t$ is the hull obtained by solving (\ref{LODE}) with
$U_t$ as driving function. Let $R(t)=\sqrt{t} + \sup_{0 \le s \le t}\{|U_s-U_0|\}$.
Then there exists a constant $c < \iy$ such that
\begin{equation}
c^{-1}R(t) \le \diam(K_t) \le c R(t).
\end{equation}
\end{Lemma}
\begin{Lemma}
\label{hcap_upper_bound}
Let $K$ be a hull. There exists a constant $c < \iy$ such that
\begin{equation}
\hcap(K) \le c \, \diam(K) \, \height(K).
\end{equation}
\end{Lemma}
\begin{proof}
Let $d = \diam(K)$ and $h=\height(K)$.
We shall only consider the case when $h < d$ (the other case is easily
verified by considering the map $z \mapsto z + d^2/z$). By scaling,
translation invariance, and
monotonicity of $\hcap$ we
may assume that $d=1$ and that $K$ is contained in the rectangle
$R=\{z: |\real z| \le 1/2, \, 0 \le \imag z \le h\}$, so that $\hcap(K)
\le \hcap(R)$. It is known that  
\begin{equation*}
\hcap(R) = \lim_{y \to \iy} y\, \EE^{iy}(\imag \,B(T)),
\end{equation*}
where $B$ denotes a complex Brownian motion and $T$ denotes the
hitting time of $\partial (\HH \setminus R)$, see \cite{Lawler_cip}. Hence  
\begin{equation*}
\hcap(R) \le \lim_{y\to \iy} y \, \omega(iy, \partial R, \HH \setminus
R) \cdot h,
\end{equation*}
where $\omega$ denotes harmonic measure. Note that $R$ can be covered by
$O(h^{-1})$ discs of radius $2h$ centered on the real line. The
harmonic measure from $iy$ of any such disc is bounded from above by
the harmonic measure of the disc centered at the origin. Since
\begin{align*}
\omega \left(iy, \partial (2h \DD), \HH \setminus (2h\DD) \right) &=(4/\pi)\arctan
\left(2h/y \right)\\
&\le c h/y
\end{align*} 
for large $y$, the lemma follows from the maximum priciple.
\end{proof}

\begin{Lemma}  \label{lemma34} 
 Suppose $f_t$ satisfies \eqref{LPDE} and  $z = x+iy\in \HH$, then 
for  $s \geq 0$ 
\begin{equation}  \label{nov19.9} 
           e^{-5as/y^2}\, |f_t'(z)|  
 \leq  |f_{t+s}'(z)|\leq 
       e^{5as/y^2}\, |f_t'(z)|  .  
\end{equation} 
In particular, if $s \leq y^2$, 
\[            e^{-5a}\, |f_t'(z)|  
 \leq  |f_{t+s}'(z)|\leq 
       e^{5a}\, |f_t'(z)|  . \] 
Also, if $0 \leq s \leq y^2$,
\begin{equation}  \label{mar14.7}
 |f_{t+s}(z) - f_t(z)| \leq \frac{y}{5}
  \, [e^{5a} - 1] \, |f_t'(z)|.
\end{equation}
\end{Lemma}

\begin{proof} Without loss of generality assume that 
$a=1$.  Differentiating \eqref{LPDE} yields 
\[    \partial_t f_t'(z)    = -f_t''(z) \, \frac{1}{z-V_t} 
            + f_t'(z) \, \frac{1}{(z-V_t)^2} . \] 
Note that $|z-V_t| \geq y$.  
Applying Bieberbach's theorem (the $n=2$ case of 
the Bieberbach conjecture) to the disk 
of radius $y$ about $z$, we can see that 
\[  |f_t''(z)| \leq  4 \, y^{-1} \, |f_t'(z)|. \] 
and hence 
\[      |\partial_t f_t'(z)| \leq 5  \, y^{-2} \, |f_t'(z)|,\] 
which implies \eqref{nov19.9}.  Returning
to \eqref{LPDE}, we see that
\[    |\p_t f_t(z)| \leq |f_t'(z)| \, \frac{a}{\imag(z)}. \]
Using \eqref{nov19.9}, we see that
\begin{eqnarray*}
 |f_{t+s}(z) - f_t(z)| & \leq & \int_0^s
   | \p_s f_{t+s}(z)| \, ds \\
    & \leq & \frac{a |f_t'(z)|}
  {y}  \int_0^s
     e^{5as/y^2} \, ds = \frac{y}{5} \, (e^{5a} - 1)\,|f_t'(z)|.
\end{eqnarray*}
\end{proof}

We shall consider Loewner chains corresponding to functions
that are H\"older continuous of order $\alpha$ for any $\alpha
<1/2$. 
We say that $U$ is \emph{weakly H\"older-1/2} if there exists a subpower function
$\vp$ such that $r^{1/2}\vp(1/r)$ is a modulus of continuity for $U$,
that is,  
\begin{equation}
\label{mod_of_cont}
\sup_{|s| \le r}|U_{t+s}-U_t| \le  r^{1/2}\vp(1/r).
\end{equation}
By P. L\'evy's theorem the sample paths of Brownian motion are almost surely weakly
H\"older-1/2 with subpower function $c\sqrt{ \log (r)}$, $c > \sqrt 2$,
see \cite[Theorem I.2.7]{Revuz_Yor}. Therefore all results for Loewner chains
corresponding to functions that satisfy \eqref{mod_of_cont} hold for $\SLE(\kappa)$ with
probability one.

\begin{Lemma}
\label{ts}
There exist constants $c,\alpha <\infty$ such that
the following holds. 
Let $(g_t)$ be a Loewner chain corresponding to the continuous
function $U_t$. Let $s \in [0, y^2]$ for $y>0$. Then
\begin{equation}
\label{ts2}
|f'_{t+s}(U_{t+s}+iy)| \le cM^{\alpha}|f'_{t
+s}(U_t+iy)|,
\end{equation}
where $M=\max\{|U_{t+s}-U_t|/y,1\}$. In particular, if $U_t$ satisfies
\eqref{mod_of_cont}, then there exists a subpower
fuction $\vp$ such that
for all $t$ and all $s \in [0,y^2]$, 
\[
|f'_{t+s}(U_{t+s}+iy)| \le \vp(1/y) |f'_{t
+s}(U_t+iy)|.
\]
\end{Lemma}
\begin{proof}
We first rescale by $|U_{t+s}-U_t|$, and then apply Lemma
\ref{whitney} with $r=\max\{|U_{t+s}-U_t|/y,1\}$ to get the conclusion.
\end{proof}
The last two lemmas immediately imply the
following result, which we record as a lemma. This essentially says
that for weakly H\"older-1/2 Loewner chains, it is enough to
consider the derivative at dyadic times.

\begin{Lemma}\label{dyadic}
Let $(g_t)$ be the Loewner chain corresponding to $U_t$ satisfying
(\ref{mod_of_cont}). Suppose that there exist constants $c$ and
$\beta$ such that for all $n \ge 1$
\begin{equation}
\label{dyadic_ass}
|\hat{f}'_{t_k}(i2^{-n})| \le c 2^{n \beta},
\end{equation}
where $t_k=k 2^{-2n}, \, k=0,1,\ldots, 2^{2n}$. Then
for every $\beta_1 > \beta$,  there exists a constant
$c_1 < \infty$   such that
\begin{equation*}
|\hat{f}'_{t}(i2^{-n})| \le c_1 2^{n \beta_1},
\end{equation*}
for $t \in [0,1]$.
\end{Lemma}
%\subsection{Upper bound}
Let $(g_t)$ be the Loewner chain corresponding to
a function $U_t$
satisfying (\ref{mod_of_cont}) that is generated
by a curve $\gamma(t)$. We want to estimate the modulus of continuity
of $\gamma$. The following quantity will be useful
\begin{equation}
v(t,y):=\int_0^y|f'_t(U_t+ir)|\, dr, \quad y > 0.
\end{equation}
The geometrical interpretation is of course the length (if it exists)
of the image of the segment $[U_t, U_t+iy]$ under $f_t$. For a given $t$, the limit
\begin{equation}\label{radial_limit}
\gamma(t)=\lim_{y \to 0+}f_t(U_t+iy)
\end{equation}
exists if $v(t,y)$ is finite
for some $y > 0$. By integration we have 
\[
|\gamma(t)-f_t(U_t+iy)| \le v(t,y).
\]
Using the Koebe one-quarter theorem, we can see that
\begin{equation}  \label{mar14.5}
           v(t,y) \geq  \, y \, |\hat f_t'(iy)|/4. 
\end{equation} 
The next result shows that Loewner chains corresponding to weakly
H\"older-1/2 functions are generated by a curve if $v(t,y)$ decays
polynomially in $y$. We also get an estimate of the modulus of
continuity of the curve.

\begin{Proposition}
\label{UB}
Let $(g_t)$ be the Loewner chain corresponding to $U_t$ satisfying
(\ref{mod_of_cont}). 
Then there exists a subpower function $\vp$ such
that if $0 \leq t \leq t +s \leq 1$ and
 $s \in [0, y^2]$
\begin{equation} \label{march15.1}
|\gamma(t+s)-\gamma(t)| \le  \vp(1/y)\left[ v(t+s, y)+v(t, y)\right].
\end{equation}
\end{Proposition}

\begin{Remark}  We have not assumed existence of the
curve in this proposition.  If $v(t+s,y),
v(t,y) < \infty$, then we know that radial limit \eqref{feb10}
exists, so we can write $\gamma(t), \gamma(t+s)$.  If
one of the radial limits does not exist we can define
$\gamma$ any way that we want since the right
hand side of \eqref{march15.1} is infinite.
\end{Remark}

\begin{proof}
We start by writing
\[
|\gamma(t+s)-\gamma(t)| \le |\gamma(t+s)-f_{t+s}(U_{t+s}+iy)| +
|\gamma(t)-f_t(U_t+iy)|\]
\[ 
 +|f_{t+s}(U_{t+s}+iy)- f_{t+s}(U_t + iy)|
   + |f_{t+s}(U_t + iy) - f_t(U_t+iy)|. 
\]
We have to estimate the last two terms. Note that
\[ |f_{t+s}(U_{t+s}+iy)- f_{t+s}(U_t + iy)|
     \leq |U_{t+s} - U_t| \, \max |f_{t+s}'(w)| , \]
where the maximum is over all $w$ on the line segment
connecting $U_{t+s} + iy$ and $U_t + iy$.  By
Lemma \ref{ts} using assumption \eqref{mod_of_cont} and then by
%\ref{whitney} and
\eqref{mar14.5}, we see that
\begin{align*}      
|f_{t+s}(U_{t+s}+iy)- f_{t+s}(U_t + iy)| &\leq c \, \vp(1/y) \, y \, |
f_{t+s}'(U_{t+s}+iy) | \\
& \leq
 c \, \vp(1/y) \, v(t+s,y).
\end{align*}
By \eqref{mar14.7} and \eqref{mar14.5}
\begin{align*}
 |f_{t+s}(U_t + iy) - f_t(U_t+iy)|
  & \leq  \frac{y}5 \,[e^{5a}-1] \, |f_t'(U_t + iy)| \\
  & \leq  c \, [e^{5a} - 1 ] \, v(t,y). 
\end{align*}
\end{proof}
\begin{Proposition}
\label{lower_holder_bound}
Let $(g_t)$ be the Loewner chain corresponding to
$U_t$ satisfying (\ref{mod_of_cont}). Suppose that $(g_t)$ is
generated by a curve $\gamma(t)$. Then for each $t$,
there exist $t_1,t_2 \in [t, t+y^2]$, and constants $c, \delta > 0$ such that 
\begin{equation*}
|\gamma(t_1)-\gamma(t_2)| \ge c \, 
|\hat{f}'_t\left(iy\right)| y \vp\left(1/y\right)^{-\delta}.
\end{equation*}
\end{Proposition}
\begin{proof}
Let $\beta:=g_t(\gamma[t, t+y^2])$. Then $\beta$ is the curve obtained by
solving (\ref{LODE}) with $U_{t+r}, \, r \in [0 , y^2],$ as driving
function. Hence $\hcap(\beta) = a y^2$, and by Lemma \ref{2.1}
together with the assumption (\ref{mod_of_cont}) on $U$,
\begin{equation}
\label{diam_beta}
\rho_d :=\diam(\beta) \le y\,\vp\left(1/y\right).
\end{equation}
Next, we combine (\ref{diam_beta}) with Lemma \ref{hcap_upper_bound} to find
\begin{equation}
\label{height_beta}
\rho_h:=
\height(\beta) \ge c \,  y \, \vp\left(1/y\right)^{-1} ,
\end{equation}
for some constant $c$.
Let $z \in \beta$ satisfy $\imag z \ge \rho_h$. Note that $|\real z-U_t|
\le \diam(\beta) \le \rho_d$. By scaling by $\rho_d^{-1}$ and then
using Lemma \ref{whitney} with $r=\vp(1/y)^2$ we get
\begin{equation}
\label{fz}
|f_t'(z)| \ge c \, \vp(1/y)^{-2\alpha}|\hat{f}'_t(iy)|,
\end{equation}
where $\alpha$ is the exponent from Lemma \ref{whitney}.
Let $w$ be a point on $\beta$ such that $|w-z| = \rho_h/2$. Since $z,w
\in \beta$ there are $t_1, t_2 \in [t, t+y^2]$ such that $\gamma(t_1)=f_t(z)$
and $\gamma(t_2)=f_t(w)$. In view of (\ref{quarter}) we have 
\begin{equation}
\label{koebe_est}
B(f_t(z), |f_t'(z)|\rho_h/16) \subset f_t(B(z, \rho_h/4)),
\end{equation}
where $B(z, \rho)$ denotes the open ball of radius $\rho$ around
$z$. Hence, using (\ref{koebe_est}) and (\ref{fz}), we conclude that
\begin{align*}
|\gamma(t_1)-\gamma(t_2)|&=|f_t(z)-f_t(w)|\\
& \ge |f_t'(z)| \, \rho_h/16 \\
& \ge c \, |\hat{f}'_t(iy)| \, y \, \vp(1/y)^{-(2\alpha+1)},
\end{align*}
and this completes the proof.
\end{proof}

Although we do not use it in this paper, we will give some results
about existence of the curve for continuous driving functions that are
not necessarily weakly H\"older-$1/2$.

\begin{Proposition}  There exists $c < \infty$
such that the following is true.  Suppose $\delta >0$
and $(g_t)$ is a Loewner chain with driving
function $U_t$.  Suppose that $0 \leq s,y< \infty$
and
\begin{equation}  \label{mar14.6}
   |U_{t+s} - U_s| \leq \delta.
\end{equation}
Then
\[      |\gamma(t+s) - \gamma(t)| 
  \leq c \, e^{5a} \, \left[v(t,\delta)
   + v(t+s,\delta)\right].\]
\end{Proposition}

\begin{proof}
We use the triangle inequality on $|\gamma(t+s) -\gamma(t)|$
as in the beginning of the proof of Proposition \ref{UB}
with $y = \delta$.
The first two terms are bounded by $v(t+s,\delta
)$ and $v(t,\delta)$,
respectively, and the fourth term is bounded in the
same way.
The distortion theorem, \eqref{mar14.6}, and \eqref{mar14.5} imply
\begin{align*}
|f_{t+s}(U_{t+s} + i\delta) - f_{t+s}(U_t + i \delta)|
    & \leq   c \, \delta \, |f'_{t+s}(U_{t+s}+ i \delta)| \\
  & \leq  c \, v(t+s,\delta), 
\end{align*}
and we get the desired estimate.
\end{proof}

\begin{Corollary}
Suppose $(g_t)$ is a Loewner chain with continuous driving
function $U_t$.  Suppose that for each $ 0 < t_1 < t_2
< \infty$, 
\[             \lim_{y \rightarrow 0+}
              v(t,y)=0  \]
uniformly for $t \in [t_1,t_2]$.  Then $(g_t)$ is generated
by a curve.
\end{Corollary}

\begin{proof}  Since $\gamma$ is a uniform limit of a sequence of
continuous functions on $[t_1,t_2]$, $\gamma$ is continuous
on $[t_1,t_2]$.  One can check directly from the Loewner
equation that $\gamma$ is right continuous at $0$ and
hence $\gamma$ is continuous.
\end{proof}

\begin{Remark}
This result also gives an estimate for the modulus of
continuity of $\gamma$.  However, the assumptions are very
strong.  If we do not assume that $U_t$ is weakly
H\"older-$1/2$, we do not have Lemma \ref{dyadic}.
\end{Remark}

\section{Proof of Theorem \ref{main.thm}}  \label{upsec}
We now turn to the proof of our main result.
The proof of Theorem \ref{main.thm} is split
into two parts: the lower bound (which requires derivative estimates
from above) is proven in Subsection \ref{main_UB}
and the upper bound (wich requires estimates from below and
control of correlations) in Subsection \ref{main_LB}.
The $\SLE$ moment estimates (Lemma \ref{nov9.prop20} and Lemma
\ref{lowerderivative}) we need for the proof are only stated
in  this section. The
proofs of the estimates build on those in \cite{Lawler_multifractal}
and are discussed in the last section.

To state the lemmas it is convenient to introduce a number of
$\kappa$-dependent parameters. 
Suppose
\[   -\infty< r < r_c := \frac 12 + \frac 4 \kappa . \]
The significance of $r_c$ is discussed in Section \ref{momsec}. 
Let
\begin{equation*}  \label{sept16.1} 
  \lambda = \lambda(r)  =  r \, \left(1 + \frac \kappa 4\right) 
   - \frac{\kappa r^2}{8},  
\end{equation*} 
\begin{equation*}  \label{sept16.2} 
  \zeta = \zeta(r) =  r - \frac{\kappa r^2}{8}, 
\end{equation*} 
and
\[          
\beta=\beta(r) = -1 + \frac{\kappa}{4 + \kappa - \kappa r}. 
\] 
Note that $\beta$ and $\lambda$ strictly increase with $r$
for $-\infty < r < r_c$ and hence, we could alternatively
consider
either of them as the free parameter.
Let $r_+$ be the larger root to $\lambda \beta + \zeta =2$
and let $\beta_+,\, \lambda_+$ and $\zeta_+$ be the corresponding
values of $\beta, \lambda$ and $\zeta$ respectively. 
Note that if $\kappa \neq 8$, then $r_+ < r_c$
and 
\[ - 1 = \beta(-\infty) < 
 \beta_+ = -1 + \frac{\kappa}{12 + \kappa - 4 \sqrt{8 + \kappa}}
 < \beta(r_c) = 1 . 
\]
Also,
\[    \alpha_* = \frac{1-\beta_+}{2}. \]

\subsection{Lower bound} \label{main_UB}
The following is the main moment estimate for the lower bound.
This was proved in \cite{Lawler_multifractal} for a certain
range of $r$ including $r=1, \kappa < 8$ which was most important
for that paper.  We give a different proof here that is valid
for all $r < r_c$. 
\begin{Lemma}  \label{nov9.prop20} 
Suppose $r < r_c$. Then there exists $c < \infty$ such that 
for all $t \geq 1$
\begin{equation}  \label{upper2} 
 \E\left[|\hat f_{t^2}'(i)|^{\lambda}\right] 
\leq c \, t^{-\zeta}. 
\end{equation} 
\begin{comment}
In particular,
\begin{equation}  \label{upper2.alt}
  \Prob\left\{|\hat f_{t^2}'(i)| \geq
      t^{\beta} \right\} \leq c \, t^{-(\zeta + \beta
 \lambda)} .
\end{equation}
\end{comment}
%\end{itemize} 
\end{Lemma} 
 
\begin{proof}  
See Section \ref{momsec}.
\end{proof} 

In the proof of this result one also finds that the expectation
in \eqref{upper2}, roughly speaking, is carried on an event
on which
$|\hat f_{t^2}'(i)| \approx  t^{\beta} $ and this has probability
of order $t^{-(\zeta + \lambda \beta)}.$
From this lemma we can derive the following uniform estimate from
which the lower bound will follow. Recall
that $\hat f'_0(iy)=1$ for all $y>0$, so we have to restrict our
attention to positive $\beta$.  
\begin{Proposition}\label{uniform_upper}  Suppose 
$\beta > \max\{0,\beta_+\}$. With probability one 
there exists $y_0 > 0$, such that for all $t \in [0,1]$ and all $y < y_0$, 
\[        |\hat f_t'(iy)| \leq y^{-\beta}. \] 
\end{Proposition}
\begin{proof}  
For $\beta > \max\{0,\beta_+\}$ we have $\lambda > 0, \zeta > 0$
and $\beta \lambda + \zeta > 2$. By choosing
$\beta $ smaller if necessary, but still larger
than  $\max\{0,\beta_+\}$, 	we can guarantee that
$\zeta < 2$.
We write 
\[
\hat{f}_{j,n}=\hat{f}_{(j-1)2^{-2n}}, \quad j=1,\ldots,2^{2n}.
\]
%\eqref{nov19.13} and \eqref{nov20.19} 
By Lemma \ref{dyadic} and the distortion theorem it 
suffices to prove that for all   
$\beta > \max\{0,\beta_+\}$, with probability 
one there exists $N< \infty$ such that for $n \geq N$, 
\begin{equation}  \label{nov14.4} 
    |\hat f_{j,n}'(i2^{-n})| \leq 2^{\beta n}, \quad
  j=1,2,\ldots,2^{2n}.  
\end{equation}  
Note that scale invariance implies 
\[  
\Prob \left \{  |\hat f_{j,n}'(i2^{-n})|\geq   2^{\beta n} 
\right\} = \Prob\left\{|\hat f_{j-1}'(i)| \geq 2^{\beta n}  
  \  \right\}.
\] 
By Lemma \ref{nov9.prop20} and Chebyshev's inequality,
\[ \Prob\left\{|\hat f_{j-1}'(i)| \geq 2^{\beta n}  
  \  \right\} \leq 2^{-\lambda \beta n} \, 
\E\left[|\hat f_{j-1}'(i)|^\lambda\right]
  \hspace{1in}\]
\[ \hspace{1in} \leq c \, j^{-\frac  \zeta 2} \,  2^{-\lambda \beta n} 
 = c \, \left(\frac{j}{2^{2n}}\right)^{-\frac \zeta 2}
    2^{-n (\lambda \beta + \zeta)}
.\]
Since $\lambda \beta + \zeta > 2$ and $\zeta < 2$,
we can sum over $j$ to get
\begin{equation}  \label{mar13.1}
     \sum_{j=1}^{2^{2n}} 
\Prob\left\{|\hat f_{j-1}'(i)| \geq 2^{\beta n}  
  \  \right\}  \leq c \, 2^{-n(\lambda \beta + \zeta - 2)} , 
\end{equation} 
and hence
\[ \sum_{n=1}^\infty \sum_{j=1}^{2^{2n}} 
\Prob\left\{|\hat f_{j-1}'(i)| \geq 2^{\beta n}  
  \  \right\}  < \infty.\]
The Borel-Cantelli lemma now implies \eqref{nov14.4}.
\end{proof} 

\begin{Proposition} \label{mar13.prop1}
Suppose $\beta_+ < \beta < 0$.
With probability one, for every $\epsilon >0$, there
exists $y_\epsilon > 0$ such that for all
$t \in [\epsilon,1]$ and all $y < y_{\epsilon}$,
\[              |f_t'(iy)| \leq y^{-\beta}. \]
\end{Proposition}

\begin{proof}  For $\beta_+ < \beta < 0$, we have
$\lambda \beta + \zeta > 2$ and $\lambda > 0$.
The proof is identical to the previous proposition
except that we do not have $\zeta >0$.  We
replace \eqref{mar13.1} with
\[  \sum_{\epsilon 2^{2n} \leq j 
 \leq 2^{2n}} 
\Prob\left\{|\hat f_{j-1}'(i)| \geq 2^{\beta n}  
  \  \right\}  \leq c_\epsilon \, 2^{-n(\lambda \beta + \zeta - 2)} .\]

\end{proof}

We can now easily prove the lower bound in Theorem \ref{main.thm}.
Recall that $\alpha_*=\min\{1/2, (1-\beta_+)/2\}$.
\begin{proof}[Proof of lower bound for Theorem \ref{main.thm}]
Let $1 \ge \beta > \tilde \beta > \max\{0, \beta_+\}$; recall that $\beta_+=1$ if and
only if $\kappa = 8$. Almost surely, by
Proposition \ref{uniform_upper}, we have for all $t \in [0,1]$
\begin{equation}\label{feb9}
v(t, y)=\int_0^{y}|\hat f_t'(ir)| \, dr \le y^{1-\tilde{\beta}}  
\end{equation}
if $y$ is small enough. The last inequality together with Proposition
\ref{UB} show that the $\SLE (\kappa)$ Loewner chain is generated by a curve
when $\kappa \neq 8$ and imply the following modulus
of continuity
\[
|\gamma(t+s)-\gamma(t)| \le c s^{(1-\beta)/2}
\]
for all $s$ small enough since $\tilde{\beta} > \beta$. The lower
bound follows.

If $\beta_+ < \beta < 0$, it suffices to prove the result
for each fixed $\epsilon$.  The argument is the same
using Proposition \ref{mar13.prop1}
\end{proof}

\subsection{Upper bound}\label{main_LB} 
In this subsection we shall use the notation
\[
\hat{f}_{j,n}=\hat{f}_{(j-1)/n^2}, \quad j=n^2/2,\ldots,n^2.
\]
\begin{Lemma}\label{lowerderivative}
Suppose  $r < r_c$. Then there exist $0 < c_1,c_2 < \infty$,
a subpower function $\phi$, and events
\[   
E_{j,n}, \;\;\;n=1,2,\ldots, \;\; j=1,\ldots,n^2 
\] 
such that the following hold.  Let $E(j,n) = 1_{E_{j,n}}$ and  
\[     
F(j,n) = n^{\zeta-2} \, |\hat f_{j,n}'(i/n)|^\lambda \, E(j,n).
\] 
\begin{itemize} 
\item{ If $n^2/2 \leq j \leq n^2$, then on the event $E_{j,n}$, 
\begin{equation}  \label{feb2.1} 
   |\hat f_{j,n}'(i/n)|  \geq \phi(n)^{-1} \, n^{\beta}.
\end{equation} }
\item{If $n^2/2 \leq j \leq n$, 
\begin{equation}  \label{feb2.1.1} 
   c_1 \, n^{-2} \leq
  \E\left[F(j,n) \right] \leq c_2\, n^{-2}.  
\end{equation}} 
\item{If $n^2/2 \leq j \leq k \leq n^2$, 
\begin{equation}  \label{feb2.1.2} 
  \E[F(j,n) \, F(k,n)] \leq  n^{-4} \, \left(\frac{n^2}{k-j+1}\right) 
     ^{\frac{\lambda \beta + \zeta }{2}} \, \phi\left(\frac{n^2}{k-j+1}\right). 
\end{equation} 
}
\end{itemize} 
\end{Lemma}
\begin{proof}
See Section \ref{lowsec}.
\end{proof}

For fixed $\beta$, we let $A_n =
 A_{n,\beta}$ denote the event 
that there exists an
 integer $j$ with $1 \leq j \leq n^2$, 
\begin{equation}  \label{nov17.8} 
       |\hat f_{j,n}'(i/n)| \geq \phi(n)^{-1} \, n^{\beta},
\end{equation} 
where $\phi$ is as in \eqref{feb2.1}
We then have the following. 
\begin{Lemma} \label{feb9.2}
Suppose $\lambda \beta + \zeta < 2$.  Then there
exist $c > 0$ such   for all $n$ sufficiently large,   
\[  
\Prob(A_{n}) \geq c . 
\] 
In particular, 
\[   
\Prob\{A_{n} \mbox{ i.o.} \} \geq c .
\]  
\end{Lemma}

\begin{proof}
Let $F(j,n)$ be as in Lemma \ref{lowerderivative},  and let
\[      Y_n = \sum_{j=n^2/2}^{n^2}  F(j,n). \]
Note that $A_n \supset \{Y_n > 0\}$.  The estimates from Lemma
\ref{lowerderivative} show that
there exist $0 < c < c_2 < \infty$ such that 
\[  \E[Y_n]  \geq c, \;\;\;\;  \E[Y_n^2] \leq c_2^2 . \]
(This uses $\beta  \lambda + \zeta < 2$.) Therefore, a
standard
second moment argument gives
\[  \Prob(A_n ) \geq
 \Prob\{Y_n > 0 \} \geq \frac{\E[Y_n]^2}{\E[Y_n^2]} 
      \geq \frac{c^2}{c_2^2}. \]
\end{proof}

We can now prove the upper bound for Theorem \ref{main.thm} and
thereby complete the proof. 
Notice that Proposition \ref{lower_holder_bound} immediately implies
that $\gamma[0,t]$ cannot be H\"older continuous of order $> 1/2$,
 since $\hat f'_0(z) =1$.
\begin{proof}[Proof of upper bound for Theorem \ref{main.thm}]
Let $\beta < \tilde \beta < \beta_+$. 
Proposition \ref{lower_holder_bound} implies that on the event
$A_n=A_{n, \tilde \beta}$ there
exist times $t_1, t_2 \in [0,1]$ such that 
\[
|\gamma(t_1)-\gamma(t_2)| \ge c n^{\beta-1}=c(n^{-2})^{(1-\beta)/2}, \quad |t_1-t_2| \le n^{-2}.
\]
Consequently on the event $\{A_n \,\, \rm{i.o}\}$, the curve $\gamma(t), \, t \in [0,1]$, is
not H\"older continuous of order $(1-\beta)/2$. To show that this
happens almost surely, let $\mathcal{A}_r$ be
the event that $\gamma(t), \, t \in [0,r]$, is not H\"older continuous
of order $(1-\beta)/2$. By Lemma \ref{feb9.2} we have
\[
\PP(\mathcal{A}_1)=:c_0>0,
\] 
and by scale invariance 
\[
\PP(\mathcal{A}_r)=c_0, \quad r >0.
\] 
Note that $\mathcal{A}_r \subset \mathcal{A}_1$ if $r \le 1$ so that 
\[
\PP \left( \bigcap_{r >0} \mathcal{A}_r\right)=c_0.
\]
Since $c_0>0$, it now follows from the Blumenthal zero-one law (see
\cite[Theorem III.2.15]{Revuz_Yor}) that
$c_0=1$ and this completes the proof. 
\end{proof}

\section{Moments of derivatives}  \label{momsec}

In this section we review some results from \cite{Lawler_multifractal}
and extend one result.  We fix $\kappa$ and we let $a = 2/\kappa$. 
 We also fix a real number $r$ such that  
\[   r < r_c = \frac{1+4a}{2} . \] 
All constants and parameters in this section depend  on 
$a$ and $r$.  We let 
\[    q = r_c - r  = 2a + \frac 12 - r > 0 . \] 
The positivity of $q$ is important for the arguments in this section 
and this is why $r$ must be less than $r_c$.

A useful tool for estimating moments 
of   $|\hat f_t'|$ is 
the reverse Loewner flow (see, e.g., 
\cite[Section 10.3]{Lawler_multifractal}). 
Suppose $U_t$ is a standard Brownian motion  
and $h_t(z)$ is the solution 
to the reverse-time Loewner equation 
\begin{equation}  \label{rlt} 
  \p_t h_t(z) = \frac{a}{U_t - h_t(z)}, \;\;\;\; 
   h_0(z) = z .  
\end{equation} 
For fixed $T$, the distribution of $h_T(z)-U_T$ is 
the same as that of $\hat f_T(z)$ and hence, $h_T'(z)$ 
has the same distribution as $\hat f_T'(z)$.  Indeed, suppose  
$g_t$ is the solution of the forward-time equation for some continuous
$V$ with $V_0=0$: 
\[      \p_t g_t(z) = \frac{a}{g_t(z) - V_t} , \] 
and we let $V^{(T)}_t = V_{T-t} - V_T$. Then if $h$ is a solution to
\eqref{rlt} with driving function $V_t^{(T)}$ we have 
\[          
g_T^{-1}(z + V_T) = h_T(z)+V_T=h_T(z)-V_T^{(T)}.
\]
It remains to note that $V_t^{(T)}, \, 0 \le t \le T$, is a standard
Brownian motion starting at $0$ if $V$ is.  
If $S < T$ and $h^{(S)}, h^{(T)}$ denote the solutions  
to \eqref{rlt} with 
driving functions $U_t^{(S)}=U_{S-t}-U_S$ and $U_t^{(T)}=U_{T-t}-U_T$ then $h^{(T)}_t, 0 
 \leq t \leq T-S$ and $h^{(S)}_t, 0 \leq t \leq S$ are 
independent.  Note that $\hat f_S'(z)\,  
\hat f_T'(z)$ has the 
same distribution as $(h_S^{(S)})'(z) \, (h_T^{(T)})'(z)$.

Let $Z_t= Z_t(i)  = X_t + i Y_t = h_t(i) - U_t$ and  
$S_t = \sin [\arg Z_t]=[1+X^2_t/Y^2_t]^{-1/2}$.  
To study the behavior of $h_t'(i)$, it is useful to 
do a time-change so that the the logarithm of 
the imaginary part grows linearly. 
Let 
\[   \sigma(t) = \inf\{s: Y_s(z) = e^{at}\}. \] 
The following lemma can be deduced from the Loewner 
equation and the analogous time change in a stochastic 
differential equation.

\begin{Lemma} \cite[Section 5]{Lawler_multifractal} 
Suppose $J_t$ satisifes  
\begin{equation}  \label{add.1} 
  dJ_t =  - r_c 
 \, \tanh J_t \, dt + d W_t, \;\;\;\; 
      J_0 =  0,  
\end{equation}
where $W_t$ is standard Brownian motion. 
Let 
\[  L_t = t - \int_0^t \frac{2\, ds}{\cosh^2 J_s} ,\] 
\begin{equation}  \label{mar11.1} 
    \sigma(t) = \int_0^t e^{2as} \, \cosh^2 J_s \; ds  , \;\;\;\; 
       \eta(s) = \sigma^{-1}(s). 
\end{equation} 
Then the joint distribution of 
\[        e^{a L_{\eta (s)}}, \;\;\;\; 
                 e^{a \eta(s)}, \;\;\;\; 
                \cosh J_{\eta(s)}, \;\;\;\;\; 0 \leq s < \infty 
 \] 
is the same as that of 
\[    |h_s'(i)| ,\;\;\;\;  Y_s , \;\;\;\; 
                 S_s^{-1} , \;\;\;\;0 \leq s < \infty. \] 
\end{Lemma}

Let $\lambda,\zeta,\beta$ be as in Section \ref{upsec}; 
  we can 
write   
\[  \lambda = r \, \left(1 + \frac{1}{2a} \right) 
  - \frac {r^2}{4a} , \;\;\;\; 
  \zeta = \lambda - \frac{r}{2a},\] 
\[     
  \beta = \frac{1-2q} {1+2q } = \frac{r-2a}{1-2r+2a}. \] 
We comment here that the $\beta$ of this paper is the same 
as $\mu$ in \cite{Lawler_multifractal}; the $\beta$ in that paper 
is half this value. 
 
Note that \eqref{add.1} can be written as 
\[ 
 dJ_t =  -(q+r)  
 \, \tanh J_t \, dt + d W_t, \;\;\;\; 
      J_0 =  0.  
\] 
Using It\^o's formula, one can see that 
\[   N_t =  e^{a \lambda L_t} \, e^{a \zeta t} \, [\cosh J_t]^r \]  
is a martingale satisfying 
\[              dN_t = r \, [\tanh J_t] \,N_t\, dW_t. \] 
Let $\Prob_*$ denote the probability measure obtained by 
weighting by the martingale $N_t$, that is, if $E$ is an event 
measurable with respect to $\{W_s: 0 \leq s \leq t\}$, 
\[           \Prob_*(E) = N_0^{-1} \, \E[N_t \, 1_E] 
             =  {\E[N_t \, 1_E]} . \]  
The Girsanov theorem (see, e.g., \cite[Chapter VIII]{Revuz_Yor} and \cite[Appendix A]{Lawler_multifractal}) implies that 
\[         dW_t = r \, \tanh J_t \, dt + dB_t , \] 
\begin{equation}  \label{add.2} 
     dJ_t = - q \, \tanh J_t \, dt + dB_t, 
\end{equation} 
where $B_t$ is a standard Brownian motion with 
respect to the measure 
$\Prob_*$.  We write $\E_*$ for expectations with respect 
to $\Prob_*$.  
 
\begin{Lemma} \cite[Lemma 7.1]{Lawler_multifractal} 
Suppose $J_t$ satisfies \eqref{add.2}. 
\begin{itemize} 
\item $J_t$ is a positive recurrent diffusion 
(with respect to the measure $\Prob_*$) with invariant density 
\[   v(x) = \frac{\Gamma(q + \frac 12)}{\Gamma(\frac 12) \Gamma(q)}  
\,  \frac{1}{\cosh^{2q} x} . \] 
\item  
\[     \int_{-\infty}^\infty \left[1 - \frac{2}{\cosh^2 x}  
\right] \,  
              v(x) \, dx =   \frac{1-2q}{1 + 2q} =  \beta. \] 
\item  Assume $J_0=0$. There exists $c < \infty$ such that if   
$k \geq 0$, $u \geq 0$,  
\begin{equation} \label{add.5} 
  \Prob_*\{\cosh J_t \geq u \mbox{ for some } k \leq t \leq 
  k+1\}  \leq c \, u^{-2q}.  
\end{equation} 
 
\end{itemize} 
\end{Lemma}

\begin{Lemma} There exists $c < \infty$ 
such that the following holds.  Suppose  
  $J_t$ satisfies \eqref{add.2} with $J_0 = x > 0.$ 
\begin{itemize} 
\item If $ y > x$,  
\begin{equation}  \label{gaussestother} 
  \Prob_*\left\{\max_{0 \leq s \leq 1} |J_s| \geq y 
  \right\}  \leq c\, \exp \left\{-\frac{(y-x)^2}{2} 
\right\}. \end{equation} 
\item If $ y < x$ and $  1 \leq t < (x-y)/q$,  
\begin{equation}  \label{gaussest} 
  \Prob_*\left\{\min_{0 \leq s \leq t} 
    J_s \leq y \right\} 
   \leq  c\, \exp\left\{ -\frac{(x -qt -y)^2}{2t} 
   \right\}.  
\end{equation} 
\end{itemize} 
\end{Lemma}

\begin{proof}  Since the drift of $J_t$ points towards the origin, 
the distribution of $|J_t|$ is stochastically dominated by 
the absolute value of a Brownian motion. Hence by the reflection
principle (see, e.g., \cite[Proposition III.3.7]{Revuz_Yor}) if $V$ is a standard
Brownian motion 
\begin{align*}
\Prob_* \left\{ \max_{0 \leq s \leq 1} |J_s| \geq y \right\} & \le
\Prob_*\left\{ \max_{0 \leq s \leq 1} |V_s| \ge y-x \right\} \\
&\le 4 \, \Prob_*\left\{V_1 \ge y-x \right\}, 
\end{align*}
which gives the first estimate.
 Also,  $J_t \geq \tilde J_t$ where $\tilde J_t$ 
satisfies 
\[             d\tilde J_t = -q \, dt + dB_t. \] 
Note that  $W_t := \tilde
J_t + tq -x$ is a standard Brownian motion starting 
at the origin. We get for $ y < x$ and $1 \leq t < (x-y)/q$, again
using the reflection principle,
\begin{align*}
 \Prob_*\left\{\min_{0 \leq s \leq t} 
    J_s \leq y \right\} & \le \Prob_* \left\{ \max_{0 \leq s \leq t}
    W_s \ge x-y-qt\right\} \\
& = 2\,\Prob_*\left\{W_t \ge x-y-qt \right\},
\end{align*}  
and the second estimate follows.
\end{proof}

\subsection{Proof of Lemma \ref{nov9.prop20} }

Lemma \ref{nov9.prop20} 
 extends a result in \cite{Lawler_multifractal} where the 
result was proved for a certain range of $r$ (including 
$r=1, a > 1/4$ which was most important for that paper).  
We start by restating it in terms of $h$ with an added lower bound (and upgrading it 
to a theorem).   
\begin{Theorem} \label{mar11.theorem}  
If $r < r_c$, there exists $0< c_1,c_2 < \iy$ such that for all $t \geq 1$, 
\[ c_1 \, t^{-\zeta} \leq 
\E\left[|h_{t^2}'(i)|^\lambda \right] 
  \leq c_2 \, t^{-\zeta}. \] 
\end{Theorem} 
 
The heuristic argument is fairly straightforward, so let 
us consider that first.  Consider the martingale $N_t$ which 
we can write as 
\[                N_t = |h_{\sigma(t)}'(i)|^\lambda 
   \, e^{at \zeta} \, S_{\sigma(t)}^{-r}. \] 
We know that $\E[N_t] = 1$.  
If $r < r_c$, then under $\Prob_*$, since $J$ is positive recurrent, $S_{\sigma(s)}$ tends to be of order $1$ for 
$s < t$.  Hence we would expect that $\sigma(t) \approx 
e^{2at}$ and hence we would expect 
\[    \E\left[|h_{e^{2at}}'(i)|^\lambda \, 
               e^{at \zeta} \right] \asymp 1. \] 
We will only prove the upper bound in Theorem \ref{mar11.theorem} 
in this subsection; 
the lower bound  
follows from the 
work of the next subsection, which uses the upper bound. 
  The next lemma gives a quantitative bound on the 
assertion $\sigma(t) \approx e^{2at}$.  As a slight abuse 
of notation, in this section, if $E$  is an event we   
also write $E$ for the indicator function of the event.

\begin{Lemma}  There exists $c < \infty$ such that for all $u > 0$, 
\begin{equation}  \label{mar11.2} 
   \Prob_*\{\sigma(t) \geq u^2 \, e^{2at} \} \leq c \, u^{-2q}.  
\end{equation} 
\end{Lemma} 
 
\begin{proof} 
 
Suppose $\cosh J_s \leq u \, e^{a(t-s)} \, (t-s+1)^{-1}$ 
for all $0 \leq s \leq t$.  Then, 
\[   \sigma(t) = \int_0^t e^{2as} \, \cosh^2 J_s \, ds 
    \leq u^2\, e^{2at} \, \int_0^t(t-s+1)^{-2} \,ds < u^2 e^{2at}. \] 
Therefore, 
\[ \Prob_* \{\sigma(t) \geq u^2 \, e^{2at}\} \leq \Prob_*\left[ 
 \bigcup_{k=1}^\infty K_k\right] \leq \sum_{k=1}^\infty 
  \Prob_*(K_k),  \] 
where $K_k= K_{k,t} 
$ denotes the event 
\begin{equation}  \label{kdef} 
 K_k = \left\{\cosh J_{t-s}   
  \geq u e^{a(k-1)} \,(k+1)^{-1} \mbox{ for some } k-1 \leq s < k\right\}. 
\end{equation} 
By \eqref{add.5}, 
\begin{equation}  \label{mar6.10} 
\Prob_*(K_k) 
  \leq c \, u^{-2q} \, e^{-2aqk} \, k^{2q}, 
\end{equation} 
and hence we can sum over  $k$ to get the result. 
\end{proof}

The next lemma is a useful ``smoothing'' result that lets us 
consider  the average of $\E[|h'_{s}(i)|^\lambda]$ over $t^2 \leq s \leq 2t^2$ 
instead of $\E[h'_{t^2}(i)|^\lambda]$. 
 
\begin{Lemma}  There exists $c < \infty$ such that 
\begin{align*} 
   \E\left[|{h_{t^2}}'(i)|^\lambda \right] 
  & \leq   \frac{c}{t^2}  \int_0^ 
     \infty e^{2as} \, \E\left[(\cosh^2J_s) 
 \,  |h_{\sigma(s)}'(i)|^\lambda \, I_s \right] \,ds \\ 
 & =    \frac{c}{t^2}  \int_0^ 
     \infty e^{as(2-\zeta)} \, \E_*\left[ 
  (\cosh J_s) ^{{2-r}}  
   \, I_s \right] \,ds , 
\end{align*} 
where $I_s = I_{s,t}$ denotes the  
event 
\[     I_s = \{ t^2 \leq \sigma(s) \leq 2\, t^2\} . \] 
\end{Lemma}

\begin{proof} 
 By scaling and the distortion theorem, if $t^2 \leq u \leq 2t^2$,  
\[  \E\left[{|h_{u}}'(i)|^\lambda \right]  
           = \E\left[{|h_{t^2}}'(it/\sqrt u)|^\lambda \right]  
 \geq c\, \E\left[|{h_{t^2}}'(i)|^\lambda \right] 
.\] 
  Therefore, 
\[   \E\left[|{h_{t^2}}'(i)|^\lambda \right] 
    \leq \frac{c}{t^2} \int_{t^2}^{2 t^2} 
    \E\left[|{h_{u}}'(i)|^\lambda \right] \, du . \]

If we let $s = \sigma^{-1} (u)$, we can change variables and write 
 \[ 
 \int_{t^2}^{2 t^2} 
     |{h_{u}}'(i)|^\lambda   \, du  
  =  \int_0^\infty |h_{\sigma(s)}'(i)|^\lambda   \,  
  I_s \, du = \int_0^\infty e^{2as} \, (\cosh^2 J_s) 
 \,  |h_{\sigma(s)}'(i)|^\lambda   \,I_s  \,ds. \] 
By taking expectations we get 
 the inequality in the lemma, and the 
equality follows from the definition of $\E_*$. 
 \end{proof} 

Since
\[  \sigma(t) \geq \int_0^t e^{2as} \, ds =
   \frac{1}{2a} \, \left[e^{2at} - 1 \right], \] 
 there is 
a $c$ such that $I_{s,e^{at}}$ is empty if $s \geq t + c$. 
Hence the  
previous proposition implies  that there exists 
a $c < \infty$ such that  
\[\E\left[|{h_{e^{2at}}}'(i)|^\lambda \right] 
  \leq  \hspace{1.5in} 
\] 
\begin{equation}  \label{add.7} 
  c \, e^{-2at} \, \int_0^{t + c} e^{as(2-\zeta)} \, \E_*\left[ 
  (\cosh J_s) ^{{2-r}}  
   \, ;\, e^{2at} \leq   
\sigma(s) \leq 2e^{2at}  \right] \,ds . 
\end{equation} 
 
\begin{Lemma}  \label{mar10.lem1} 
There exists $c < \infty$ such that if $u \geq 1$,  
\[  \E_*[(\cosh J_t)^{{2-r}} ; 
 u^2  \leq e^{-2at} \, \sigma(t) \leq 2u^2] 
 \leq c \, u^{(2-r)_+ - 2q}. \]  
\end{Lemma} 
 
\begin{Remark}  Roughly speaking, we expect that if 
$\sigma(t) \asymp u^2 \, e^{2at}$, then $\cosh J_t 
\approx u$.  Hence, we would guess 
\[  \E_*[(\cosh J_t)^{{2-r}} ; 
 u^2  \leq e^{-2at} \, \sigma(t) \leq 2u^2] 
\approx \hspace{1in} \] 
\[ \hspace{1in}  u^{2-r} \, \Prob_*\{ 
u^2  \leq e^{-2at} \, \sigma(t) \leq 2u^2  \} 
 \leq c \, u^{2-r-2q} , \] 
where the probability is estimated by \eqref{mar11.2}. 
This lemma makes the argument rigorous but  
only gives a weaker result 
  for $r > 2$.  Lemma \ref{mar10.lem2} 
below establishes the stronger result for some values 
of $r >2$. 
\end{Remark}

\begin{proof}   
Let $I_t = I_{t,u}$ be the  the event 
$\{u^2  \leq e^{-2at} \, \sigma(t) \leq 2u^2\}.$ 
We claim that $I_t$ is contained in the event 
\[    A = A_{t,u}  = \left\{ 
\min_{(t-1) \vee 0 \leq s \leq t } 
\cosh^2 J_s \leq  2u^2e^{2a}  \right\} 
 \] 
Indeed, this is obvious for $t \leq 1$  since $\cosh J_0 =1$ 
and if $t >1$ and $\cosh^2 J_s > 2e^{2a}u^2$ for $t-1 
\leq s \leq t$, then 
 \[  \sigma(t) \geq \int_{t-1}^t e^{2as} \, \cosh^2 J_s 
  \, ds  > 2u^2\, e^{2at}. \]

Let 
\[      V_k = V_{k,t,u} 
 = I_t \cap \left\{   u \, e^{(k-1)} \leq \cosh J_t < 
                 u\,  e^{k} \ 
\right\} . \]  
Then, 
\[  \E_*[ (\cosh J_t)^{{2-r}}  \, 
    I_t] = \sum_{k=-\infty}^\infty  \E_*[ (\cosh J_t)^{{2-r}} 
   \, 
    V_k]. \] 
Note that 
\begin{equation}  \label{mar11.4} 
\E_*[ (\cosh J_t)^{{2-r}} 
   \, 
    V_k] \asymp u^{2-r} \, e^{k(2-r)} \, \Prob_*(V_k). 
\end{equation}

We will first show that  
\begin{equation}  \label{mar10.3} 
  \sum_{k=1}^\infty  \E_*[ 
(\cosh J_t)^{{2-r} } \, 
    V_k] \leq c \, u^{2-r - 2q} .  
\end{equation} 
Let $k_0$ be an integer such that $e^{k-1} 
> 2e^{2a}$ for $k \geq k_0$. Then, 
\begin{align*}  \sum_{k=1}^{k_0}  \E_*[ 
(\cosh J_s)^{{2-r} } \, 
    V_k] & \asymp   
   u^{2-r}   \sum_{k=1}^{k_0}   
\Prob_*(V_k)\\ & 
 \leq  u^{2-r} \, \Prob_*\{ 
     \cosh^2 J_t \geq u^2\} 
  \leq c  
  u^{2-r-2q}.  
\end{align*} 
The last inequality uses \eqref{add.5}. 
 
If $k > k_0$,  let  \[\eta_t = \inf\left \{s \geq (t-1) \vee 0: 
  \cosh^2 J_s =  2e^{2a}u^2 \right\}. \]  
Since $V_k \subset A$ and $\cosh^2 
J_t >  2e^{2as}u^2 $ on $V_k$, we know that 
on the event $V_k$,  $t-1 \leq \eta_t 
 < t$. Hence we can estimate 
\[   \Prob_*(V_k) \leq \Prob_*\{\eta_t < t\} 
 \, \Prob_*\left\{\cosh^2 J_t \geq u^2 \, e^{k-1} \; \mid 
  \; \eta_t < t \right\}.\] 
 By \eqref{add.5},  
\[ \Prob_*\{\eta_t < t\} \leq \Prob_*\left\{ 
 \cosh^2 J_s \geq 2e^{2a} u^2 \mbox{ for some } 
 t-1 \leq s \leq t \right\}  
 \leq c \, u^{-2q}.\] 
Using \eqref{gaussestother}, we can see that there 
exist $c,\alpha$ such that  
\[  \Prob_*\left\{\cosh^2 J_t \geq u^2 \, e^{k-1} \; \mid 
  \; \eta_t < t \right\} \leq c \, e^{-\alpha k^2}. \] 
 Hence, plugging into \eqref{mar11.4}, we 
have  
\[ \E_*[ (\cosh J_t)^{{2-r}} 
   \, 
    V_k]  \leq c \, u^{2-r-2q} \, e^{k(2-r)} \, e^{- 
 \alpha k^2}. \] 
Therefore, 
\[  u^{2q+r-2} \,  
 \sum_{k=1}^\infty  \E_*[ 
(\cosh J_t)^{{2-r} } \, 
    V_k] \leq c \sum_{k=0}^\infty  
 e^{k(2-r)} \, e^{- 
 \alpha k^2 < \infty}. \] 
This proves \eqref{mar10.3}. 
 
Let  
$  \tilde V = \bigcup_{k=0}^\infty  
   V_{-k} . $ 
On the event $\tilde  V$, 
\[     1 \leq \cosh J_t \leq u . \] 
Hence $(\cosh J_t)^{2-r} \leq u^{(2-r)_+}$ and, 
using \eqref{mar11.2},  
\[ \E_*[ (\cosh J_t)^{{2-r}} 
   \, 
    \tilde V] \leq u^{(2-r)_+} 
  \, \Prob_*(\tilde V) \leq  
    u^{(2-r)_+} \, \Prob_*(I_t) \leq c \, 
  u^{(2-r)_+} \, u^{-2q}. \] 
\end{proof}

\begin{Lemma}  \label{mar10.lem2} 
If $r > a+1$, there exists $c < \infty$ such that 
for all $u$,  
\[  \E_*[(\cosh J_t)^{{2-r}} ; 
 u^2  \leq e^{-2at} \, \sigma(t) \leq 2u^2] 
 \leq c \, u^{2-r - 2q}. \]  
\end{Lemma} 
 
\begin{proof} 
 
We use the notation from the proof of Lemma \ref{mar10.lem1} 
and let $V^j = V_{-j}$.  We 
note that \eqref{mar10.3} holds for all values of $r$.  Hence, 
we only need to show 
\[  \sum_{j=0}^\infty \E_*[ 
(\cosh J_t)^{{2-r} } \, 
    V^j] \leq c \, u^{2-r - 2q} .\]

Let 
$K_k$ be as in \eqref{kdef}. 
Then 
\[ \sum_{j=0}^\infty \E_*[ 
(\cosh J_t)^{{2-r} } \, 
    V^j] \leq \sum_{k=1}^\infty 
   \sum_{j=0}^\infty    \E_*[(\cosh J_t)^{2-r} \, 
       V^j \, K_k]. \] 
Note that 
\begin{align*} 
 u^{2q+r-2} \, \E_*[(\cosh J_t)^{2-r} \, 
          V^j \, K_k] & \asymp   u^{2q} e^{j(r-2)} 
  \, \Prob_*(V^j \cap K_k)\\&  =   
       u^{2q} \, \Prob_*(K_k) \, e^{j(r-2)} 
  \, \Prob_*(V^j \mid K_k).  
\end{align*} 
By \eqref{mar6.10}, we know that 
$u^{2q} \, \Prob_*(K_k) \leq c \, e^{-2aqk} \, k^{2q}$.  
Hence, it suffices to prove that 
\begin{equation}  \label{mar10.1} 
 \sum_{k=1}^\infty  e^{-2aqk} \, k^{2q} \sum_{j=0}^\infty 
        e^{j(r-2)} \, \Prob_*(V^j\mid K_k) < \infty .  
\end{equation}

Note that for $1+ a < r < r_c$, we have $a > q$ and  
  \[  -2 a q + \frac{(r-2)^2}{2}  +(2-r) \, 
             (a - q) < 0, \] 
Choose $\hat a$ satisfying $q < \hat a < a$ and 
\begin{equation} \label{arange} 
 -2 \hat a q + \frac{(r-2)^2}{2}  +(2-r) \, 
             (\hat a - q) < 0, 
\end{equation} 
 We will show  
\[  \sum_{k=1}^\infty  e^{-2\hat aqk}   \sum_{j=0}^\infty 
        e^{j(r-2)} \, \Prob_*(V^j\mid K_k) < \infty . \] 
On the event $K_k$, let $s$ be the largest number less than or equal 
to $k$ such that  
\[  \cosh J_{t-s}   
  \geq u e^{a(k-1)} \,(k+1)^{-1}. \] 
By the definition of $K_k,$ we know that $k-1 \leq  s \leq k$.  Also, 
we can find $c_1$ such that  
\[      J_{t-s} \geq \log u + a(k-1) - \log (k+1) 
  \geq \log u + \hat a k - c_1 + \log 2.\] 
On the event $V^j$, $\cosh J_t \leq e^{-j} u$ 
which implies 
\[        J_t  \leq \log u  -j  + \log 2. \] 
We estimate  $\Prob_*(V^j \mid K_k)$ from above by 
\[   \sup_{k-1 \leq s \leq k} 
            \Prob_*\left\{J_t \leq  \log u  -j + \log 2 
   \mid J_{t-s} =  \log u + \hat ak - c_1 + \log 2 \right\}. \] 
As in \eqref{gaussest}, this probability is bounded above by the corresponding 
probability if $J_t$ is a Brownian motion with drift $-q$.  
In particular, \eqref{gaussest} shows that 
there exists $M$ and $c_2$ such that 
for $j > M$, 
\begin{equation}  \label{mar10.5} 
 \Prob_*(V^j \mid K_k)  \leq 
 c \, \exp\left\{ - \frac{1}{2k}  \, [(\hat 
 a-q)k  +j - c_1]^2  
  \right\} . 
\end{equation} 
Note that 
\[  \sum_{k=1}^\infty  e^{-2\hat aqk}   \sum_{j=0}^{M } 
        e^{j(r-2)} \, \Prob_*(V^j\mid K_k) 
 \leq  \sum_{k=1}^\infty  e^{-2\hat aqk}   \sum_{j=0}^{M } 
        e^{j(r-2)}  
 < \infty . \] 
Also, \eqref{mar10.5} gives 
\[  
 \sum_{j=M+1}^{\infty } 
        e^{j(r-2)} \, \Prob_*(V^j\mid K_k) 
  \leq c \sum_{j=M+1}^\infty  
  \exp\left\{ j(r-2) - \frac{1}{2k}  \, [j+b]^2\right\} .\] 
where 
\[    b  = (\hat a-q)k  - c_2.\] 
We bound the right hand side from above by a constant 
times  
\begin{align*}     \int_{-\infty}^\infty e^{x(r-2)} 
               e^{-\frac{(x+b)^2}{2k}} \, dx &  
  =    e^{b(2-r)} \int_{-\infty}^\infty e^{y(r-2)} 
               e^{-\frac{y^2}{2k}} \, dy \\ 
& \leq  c \,  \sqrt k \,  \exp\left\{ b(2-r)   
 +  \frac{k(r-2)^2}{2} \right\} \\ 
 & \leq c \, \sqrt k \, \exp\left\{ (2-r)(\hat a - q)k   
 +  \frac{k(r-2)^2}{2} \right\}  
\end{align*} 
Hence 
\[   e^{-2\hat a qk} \,  \sum_{j={M+1}}^\infty  
  e^{j(r-2)}  
   \Prob_*(V^j \mid K_k) \leq C_\epsilon \, \sqrt k 
   \, e^{k \xi}\] 
where  
\[ \xi =  
 - 2 \hat a q   + (2-r)(\hat a - q)   
 +  \frac{ (r-2)^2}{2} .\] 
Recalling from  
 \eqref{arange} that $\xi < 0$, we conclude 
\[ \sum_{k=1}^\infty 
  e^{-2\hat a qk} \,  \sum_{j={M+1}}^\infty  
  e^{j(r-2)}  
   \Prob_*(V^j \mid K_k) <\infty . \] 
 \end{proof} 
 
\begin{Lemma}  Suppose $r < r_c$.  Then there exists $\theta 
< 2-\zeta$ and $c < \infty$ such that for all $u$ 
\begin{equation}  \label{mar10.6} 
  \E_*[(\cosh J_t)^{{2-r}} ; 
 u^2  \leq e^{-2at} \, \sigma(t) \leq 2u^2] 
 \leq c \, u^{\theta}.  
\end{equation} 
\end{Lemma}

\begin{proof}  We set 
\[   \theta = \left\{\begin{array}{ll} 2-r-2q, & \mbox{ if } 
    r \leq 2 \mbox{ or } r \geq a + \frac 12 \\ 
        -2q & \mbox {otherwise}.  
\end{array} \right. \] 
The estimate \eqref{mar10.6} follows from the previous two 
lemmas so we only need to verify that $\theta < 2 - \zeta$. 
It is easy to see that $1-r-2q < 2 - \zeta$ for all $r$.  Also, 
one can 
check that $  -2q < 1-4a + r $ provided that $a < 3/5$ or 
$a \geq 3/5$ and  
\begin{equation}  \label{crit} 
            r < 2a \, \left[3 - \sqrt{5 - (3/a)}\right].  
\end{equation} 
If $a \leq 4$, then one can show that \eqref{crit} holds 
for all $r < r_c$.  Hence we only need to consider $a \geq 4$ 
and $r < r_c$ that do not satisfy \eqref{crit}.  Such an 
$r$ satisfies 
\[ r \geq  2a \, \left[3 - \sqrt{5 - (3/a)}\right] 
    \geq 2a \, [3 - \sqrt{17/4}] \geq  9a/5 . \] 
In particular, $r > a+1$. 
\end{proof}

\begin{proof}[Proof of Theorem \ref{mar11.theorem}]  
Let $\theta$ be as in the previous lemma.  Then by 
\eqref{mar10.6},  
\[ 
\E_*\left[ 
  (\cosh J_s) ^{{2-r}}  \, ;\,  
    e^{2at} \leq   
\sigma(s) \leq 2e^{2at}  \right]  
\leq  ce^{(t-s)a \theta}. 
\] 
Therefore, by  
 \eqref{add.7},  
\begin{align*} 
 \E\left[|h_{e^{2at}}'(i)|^\lambda   
  \right] 
 &\leq  c \, e^{-2at} \, \int_{0}^{t+ c} \,  
e^{2as} \, e^{-a s \zeta} \, e^{(t-s) a\theta} \,  
  ds  \\ 
 & \leq  c \, e^{-a\zeta t} \, \int_{0}^{t+c} 
e^{a(t-s)[-2+ \zeta + \theta]}  
   \,ds.\\ 
 & \leq    c \, e^{-a\zeta t} \, \int_{-c}^{\infty} 
e^{ay[-2+ \zeta + \theta]}  
   \,dy \leq c \, e^{-a\zeta t}.\\ 
\end{align*} 
The last inequality uses 
$  \theta < 2 - \zeta. 
  $

\end{proof}

\subsection{Proof of Lemma 
\ref{lowerderivative}}  \label{lowsec} 
 
We essentially follow the proof in \cite{Lawler_multifractal}. 
In that paper, it was assumed that $r \geq 0$, which we do 
not want to assume here, but with the upper bound of 
Theorem \ref{mar11.theorem}, we can do the argument for 
all $r < r_c$.  We emphasize that the postive recurrence 
of $J_t$ is critical for this argument and hence we need 
$q = r_c - r > 0$. 
Since $J_t$ is positive recurrent, we expect that $J_t = O(1)$ 
and that (approximately) 
\[    L_t =  \beta t + O(t^{1/2}). \] 
The next lemma gives an estimate of this type.

\begin{Lemma}\cite[Proposition 7.3]{Lawler_multifractal} 
  For each $u,t > 0$, let $E_{t,u}$ be the event that 
the following holds for all $0 \leq s \leq t$: 
\[           |J_s| \leq u \, \log\left[\min\{s+2 , t-s+2 \}\right], \] 
\[      |L_s - s\beta| \leq u \, \sqrt s \, \log(s+2) , \] 
\[       |L_t - L_s - (t-s) \beta| \leq 
            u \, \sqrt{t-s} \, \log(t-s+2) . \] 
Then, 
\[            \lim_{u \rightarrow \infty} 
                   \inf_{t > 0} \Prob_*(E_{t,u}) = 1 . \] 
\end{Lemma}

By integrating, we can see there exists $c_0$ and 
$ c(u)$ such that for all 
$t \geq 1$, on the event  $E_{t,u}$ 
\begin{equation}  \label{add.6} 
     c_0 \, e^{2at} \leq  \sigma(t) \leq c(u) \, e^{2at}.  
\end{equation} 
Note that the lower bound does not depend on $u$ and follows from 
the estimate $\cosh^2 J_s \geq 1$.  An event such as this 
is used to define the event $E_t = E_{t,u}$ in Lemma 4.2 where 
$u$ is sufficiently large so that 
\[    \Prob_*(E_t) = \E\left[N_t \, 1_{E_t}\right] \geq \frac 
 12. \] 
Note that on the event $E_t$,  
\[                |h_{\sigma(t)}'(i)| \asymp 
                       e^{at \beta}, \] 
and using \eqref{add.6}, one can show that $|h_{e^{2at}}'(i)| 
  \asymp 
 (e^{at})^\beta$.  In fact, there is a  function 
$\psi$ with $\psi(s) = o(s)$ as  
$s \rightarrow \infty$, 
 such that for all $0 \leq s \leq t$, 
\[          e^{a \beta s - \psi(s)} 
 \leq       |h_{\sigma(s)}'(i)| \leq e^{a\beta s +\psi(s)}, \] 
\[   e^{a \beta(t-s) - \psi(t-s)} \leq 
    \left| \frac{h_{\sigma(t)}'(i)}{h_{\sigma(s)}'(i)}\right| 
  \leq e^{a\beta(t-s) + \psi(t-s)}.\] 
As in \eqref{add.6}, we can see that this implies (with 
perhaps a different $o(s)$ function $\psi$) 
 \[          e^{a \beta s - \psi(s)} 
 \leq       |h_{e^{2as}}'(i)| \leq e^{a\beta s +\psi(s)}, \] 
\[   e^{a \beta(t-s) - \psi(t-s)}\,  
\left|{h_{e^{2as}}'(i)}\right| 
  \leq 
    \left| {h_{e^{2at}}'(i)}\right| 
  \leq e^{a\beta(t-s) + \psi(t-s)}\,  
\left|{h_{e^{2as}}'(i)}\right| 
.\] 
We get 
\begin{Proposition} \label{mar12.prop1} 
 If $r < r_c$, there is a $c< \infty$ 
and a  subpower function 
$\phi$ such that if $E_t = E_{t,\beta,\psi}$ denotes the event that for 
$1 \leq s \leq t$,  
\begin{equation}  \label{mar12.1} 
               s^\beta \, \phi(s)^{-1} 
  \leq   |h'_{s^2}(i)| \leq  s^\beta \, \phi(s) ,  
\end{equation} 
\begin{equation}  \label{mar12.2} 
      (t/s)^\beta \, \phi(t/s)^{-1} 
      \, |h'_{s^2}(i)| \leq |h_{t^2}'(i)| \leq (t/s)^\beta \,  
   \phi(t/s)  \, |h'_{s^2}(i)|  ,  
\end{equation} 
then 
\[                  \E\left[|h_{t^2}'(i)|^\lambda 
  \, 1_{E_t}\right]  \geq c . \] 
\end{Proposition} 
 
Note that  
\[       |h_{t^2}'(i)|^\lambda 
  \, 1_{E_t} \leq |h_{s^2}'(i)|^\lambda  \, (t/s)^{\beta\lambda}   
  \, \phi(t/s)^\lambda.\]    
By the definition of the event, 
we also get $Y_{e^{2at}} \asymp e^{at}$ and $S_{e^{2at}} \asymp 
1$, so that 
\[            \E\left[|h_{e^{2at}}'(i)|^\lambda \right] 
      \asymp [e^{ at}]^{-\zeta}. \] 
This is how \eqref{feb2.1} and \eqref{feb2.1.1} are derived. 

%To get \eqref{feb2.1.2} one notes that 
%correlations at times $s$ and $s+t$ for $\hat f$ correspond to 
%correlations at times $s$ and $s+t$ for $h$. 
 
To get \eqref{feb2.1.2} we have to estimate the
 correlations. Suppose that $S,T$ are positive 
integers with $n^2/2 \leq S \leq  T \leq n^2$.   Let 
$h^{(S)},  h^{(T)}$ be defined as in the beginning of
the section. Recall that $\hat f_S'(z) \, \hat f'_T(z)$ has the same
distribution as $(h_S^{(S)})'(z)(h_T^{(T)})'(z)$. Set $h:=h^{(S)}$ and
$\tilde h:=h^{(T)}$. Let $E(S) , E(T)$ denote 
the indicator function of the event from Proposition  
\ref{mar12.prop1} with $t=n$ and with $h$ being 
$h$ or $\tilde h$, respectively.  We have already noted 
that $\tilde h_t, 0 \leq t \leq T-S$ is independent 
of $h_s, 0 \leq s \leq S$.   
Using \eqref{mar12.2} we see that 
\[     |\tilde h_{T}'(i)|^\lambda \, E(T) 
    \leq |\tilde h_{T-S}'(i)|^\lambda   
  \,  \left(\frac{n^2}{T-S+1}\right)^{\frac {\beta\lambda} 
   2} \, 
   \phi\left(\frac{n^2}{T-S+1}\right). \] 
(Here we use $\phi$ for a subpower function whose exact 
value may change from line to line.) 
The random variable on the right hand side is independent of 
$h$.  Therefore, using Theorem \ref{mar11.theorem}, 
\begin{align*} 
\lefteqn{ \E\left[|h_{S}'(i)|^\lambda \, |\tilde h_{T}'(i)| 
    ^\lambda\, E(T) \right]} \\ 
 & \leq         \E\left[|h_{S}'(i)|^\lambda \right] 
  \, \E\left[|\tilde h_{T-S}'(i)|^\lambda \right] \, 
      \left(\frac{n^2}{T-S+1}\right)^{\frac{\beta\lambda}2} \, 
   \phi\left(\frac{n^2}{T-S+1}\right)\\ 
& \leq  c \, n^{-\zeta} \, (T-S+1)^{-\zeta/2} 
  \,  \left(\frac{n^2}{T-S+1}\right)^{\frac{\beta\lambda}2} \, 
   \phi\left(\frac{n^2}{T-S+1}\right)\\ 
& =  c \, n^{-2\zeta} \,  
 \left(\frac{n^2}{T-S+1}\right)^{\frac{ 
 \beta\lambda+ \zeta} 2}  
\phi\left(\frac{n^2}{T-S+1}\right),
\end{align*} 
and \eqref{feb2.1.2} follows.

\bibliography{JL}
\end{document}